\documentclass[a4paper,twoside]{amsart}
\usepackage[utf8]{inputenc}
\usepackage{amssymb}

\def\str#1{\mathbf {#1}}

\def\Aut{\mathop{\mathrm{Aut}}\nolimits}
\def\Alt{\mathop{\mathrm{Alt}}\nolimits}

\usepackage{graphicx}
\newcommand{\overbar}[1]{\mkern 1.5mu\overline{\mkern-1.5mu#1\mkern-1.5mu}\mkern 1.5mu}

\theoremstyle{plain}
\newtheorem{theorem}{Theorem}[section]

\newtheorem{lemma}[theorem]{Lemma} 

\newtheorem{question}[theorem]{Question}

\theoremstyle{definition}

\newtheorem*{remark*}{Remark}
\newtheorem*{claim*}{Claim}

\theoremstyle{remark}

\begin{document}
\title{Ramsey expansions of 3-hypertournaments}

\authors{
\author[G. Cherlin]{Gregory Cherlin}
\address{Department of Mathematics, Rutgers University,
110 Frelinghuysen Rd., Piscataway, NJ 08854 USA}
\email{cherlin.math@gmail.com}
\author[J. Hubi\v cka]{Jan Hubi\v cka}
\address{Charles University, Faculty of Mathematics and Physics\\Department of Applied Mathematics (KAM)\\
Prague, Czech Republic}
\email{hubicka@kam.mff.cuni.cz}
\author[M. Kone\v cn\'y]{Mat\v ej Kone\v cn\'y}
\address{Charles University, Faculty of Mathematics and Physics\\Department of Applied Mathematics (KAM)\\
Prague, Czech Republic}
\email{matej@kam.mff.cuni.cz}
\author[J. Ne\v set\v ril]{Jaroslav Ne\v set\v ril}
\address{Charles University, Faculty of Mathematics and Physics\\
Computer Science Institute of Charles University (IUUK)\\
Prague, Czech Republic}
\email{nesetril@iuuk.mff.cuni.cz}

\thanks{This paper is part of a project that has received funding from the European Research Council (ERC) under the European Union’s Horizon 2020 research and innovation programme (grant agreement No 810115). Jan Hubi\v cka and Mat\v ej Kone\v cn\'y are further supported by project 21-10775S of the Czech Science Foundation (GA\v CR), Jan Hubi\v cka is also supported by Center for Foundations of Modern Computer Science (Charles University project UNCE/SCI/004) and Mat\v ej Kone\v cn\'y is also supported by the Charles University Grant Agency (GA UK), project 378119.}
}


\begin{abstract}
We study Ramsey expansions of certain homogeneous 3-hypertournaments. We show that they exhibit an interesting behaviour and, in one case, they seem not to submit to current gold-standard methods for obtaining Ramsey expansions. This makes these examples very interesting from the point of view of structural Ramsey theory as there is a large demand for novel examples.

\end{abstract}

\maketitle

Structural Ramsey theory studies which homogeneous structures have the so-called Ramsey property, or at least are not far from it (can be \emph{expanded} by some relations to obtain a structure with the Ramsey property). Recently, the area has stabilised with general methods and conditions from which almost all known Ramsey structures follow. In particular, the homogeneous structures offered by the classification programme are well-understood in most cases. Hence, there is a demand for new structures with interesting properties.

In this abstract we investigate Ramsey expansions of four homogeneous 4-constrai\-ned 3-hypertournaments identified by the first author~\cite{Cherlinhypertournaments} and show that they exhibit an interesting range of behaviours. In particular, for one of them the current techniques and methods cannot be directly applied. There is a big demand for such examples in the area, in part because they show the limitations of present techniques, in part because they might lead to a negative answer to the question whether every structure homogeneous in a finite relational language has a Ramsey expansion in a finite relational language, one of the central questions of the area asked in 2011 by Bodirsky, Pinsker and Tsankov~\cite{Bodirsky2011a}.

\section{Preliminaries}
We adopt the standard notions of languages (in this abstract they will be relational only), structures and embeddings. A structure is \emph{homogeneous} if every isomorphism between finite substructures extends to an automorphism. There is a correspondence between homogeneous structures and so-called \emph{(strong) amalgamation classes} of finite structures, see e.g.~\cite{Hodges1993}. A structure $\str A$ is \emph{irreducible} if every pair of vertices is part of a tuple in some relation of $\str A$.

In this abstract, an \emph{$n$-hypertournament} is a structure $\str A$ in a language with a single $n$-ary relation $R$ such that for every set $S\subseteq A$ with $|S| = n$ it holds that the automorphism group of the substructure induced on $S$ by $\str A$ is precisely $\Alt(S)$, the alternating group on $S$. This in particular means that exactly half of $n$-tuples of elements of $S$ with no repeated occurrences are in $R^\str A$. For $n=2$ we get standard tournaments, for $n=3$ this correspond to picking one of the two possible cyclic orientations on every triple of vertices. It should be noted however, that another widespread usage, going back at least to Assous~\cite{Assoushypertournaments}, requires a unique instance of the relation to hold on each $n$-set. A \emph{holey $n$-hypertournament} is a structure $\str A$ with a single $n$-ary relation $R$ such that all irreducible substructures of $\str A$ are $n$-hypertournaments. A \emph{hole} in $\str A$ is a set of 3 vertices on which there are no relations at all.

Let $\str A,\str B, \str C$ be structures. We write $\str C\longrightarrow (\str B)^\str A_2$ to denote the statement that for every 2-colouring of embeddings of $\str A$ to $\str C$, there is an embedding of $\str B$ to $\str C$ on which all embeddings of $\str A$ have the same colour. A class $\mathcal C$ of finite structures has the \emph{Ramsey property} (\emph{is Ramsey}) if for every $\str A,\str B\in \mathcal C$ there is $\str C\in \mathcal C$ with $\str C\longrightarrow (\str B)^\str A_2$ and $\mathcal C^+$ is a \emph{Ramsey expansion} of $\mathcal C$ if it is Ramsey and can be obtained from $\mathcal C$ by adding some relations. By an observation of Ne\v set\v ril~\cite{Nevsetril2005}, every Ramsey class is an amalgamation class under some mild assumptions.

\subsection{Homogeneous 4-constrained 3-hypertournaments}
Suppose that $\str T = (T,R)$ is a 3-hypertournament and pick an arbitrary linear order $\leq$ on $T$. One can define a 3-uniform hypergraph $\hat{\str{T}}$ on the set $T$ such that $\{a,b,c\}$ with $a\leq b\leq c$ is a hyperedge of $\hat{\str T}$ if and only if $(a,b,c)\in R$. (Note that by the definition of a 3-hypertournament, it always holds that exactly one of $(a,b,c)$ and $(a,c,b)$ is in $R$.) This operation has an inverse and hence, after fixing a linear order, we can work with 3-uniform hypergraphs instead of 3-hypertournaments. There are three isomorphism types of 3-hypertournaments on 4 vertices:
\begin{description}
    \item[$\str H_4$] The homogeneous 3-hypertournament on 4 vertices. For an arbitrary linear order $\leq$ on $H_4$, $\hat{\str H_4}$ contains exactly two hyperedges. Moreover, they intersect in vertices $a<b$ such that there is exactly one $c\in H_4$ with $a < c < b$.
    \item[$\str O_4$] The odd 3-hypertournament on 4 vertices. For an arbitrary linear order $\leq$, $\hat{\str O_4}$ will contain an odd number of hyperedges. Conversely, any ordered 3-uniform hypergraph on 4 vertices with an odd number of hyperedges will give rise to $\str O_4$.
    \item[$\str C_4$] The cyclic 3-hypertournament on 4 vertices. There is a linear order $\leq$ on $C_4$ such that $\hat{\str C_4}$ has all four hyperedges. In other linear orders, $\hat{\str C_4}$ might have no hyperedges or exactly two which do not intersect as in $\str H_4$.
\end{description}

We say that a class $\mathcal C$ of finite 3-hypertournaments is \emph{4-constrained} if there is a non-empty subset $S\subseteq \{\str H_4, \str O_4, \str C_4\}$ such that $\mathcal C$ contains precisely those finite 3-hypertournaments whose every substructure on four distinct vertices is isomorphic to a member of $S$. There are four 4-constrained classes of finite 3-hypertournaments which form a strong amalgamation class~\cite{Cherlinhypertournaments}. They correspond to the following sets $S$:

\begin{description}
    \item[$S=\{\str C_4\}$] The \emph{cyclic} ones. These can be obtained by taking a finite cyclic order and orienting all triples according to it. Equivalently, they admit a linear order such that the corresponding hypergraph is complete.

    \item[$S=\{\str C_4,\str H_4\}$] The \emph{even} ones. The corresponding hypergraphs satisfy the property that on every four vertices there are an even number of hyperedges.

    \item[$S=\{\str C_4,\str O_4\}$] The \emph{$\str H_4$-free} ones. Note that in some sense, this generalizes the class of finite linear orders: As $\Aut(\str H_4) = \Alt(4)$, one can define $\str H_n$ to be the $(n-1)$-hypertournament on $n$ points such that $\Aut(\str H_n)=\Alt(n)$. For $n=3$, we get that $\str H_3$ is the oriented cycle on 3 vertices and the class of all finite linear orders contains precisely those tournaments which omit $\str H_3$.

    \item[$S=\{\str C_4,\str O_4,\str H_4\}$] The class of all finite 3-hypertournaments. 
\end{description}

\section{Positive Ramsey results}
In this section we give Ramsey expansions for all above classes with the exception of the $\str H_4$-free ones. Let $\mathcal C_{c}$ be the class of all finite cyclic 3-hypertournaments. Let $\overrightarrow{\mathcal C_{c}}$ be a class of finite linearly ordered 3-hypertournaments such that $(A,R,\leq) \in \mathcal C_{c}$ if and only if for every $x < y < z \in A$ we have $(x,y,z)\in R$. Notice that for every $(A,R)\in \mathcal C$ there are precisely $|A|$ orders $\leq$ such that $(A,R,\leq)\in\overrightarrow{\mathcal C_{c}}$ (after fixing a smallest point, the rest of the order is determined by $R$), and conversely, for every $(A,R,\leq)\in \overrightarrow{\mathcal C_{c}}$ we have that $(A,R)\in\mathcal C_{c}$.

\medskip

It is a well-known fact that every Ramsey class consists of linearly ordered structures~\cite{Kechris2005}. We have seen that after adding linear orders freely, the class of all finite ordered even 3-hypertournaments corresponds to the class of all finite ordered 3-uniform hypergraphs which induce an even number of hyperedges on every quadruple of vertices. These structures are called \emph{two-graphs} and they are one of the reducts of the random graph (one can obtain a two-graph from a graph by putting hyperedges on triples of vertices which induce an even number of edges). Ramsey expansions of two-graphs have been discussed in~\cite{eppatwographs} and the same ideas can be applied here.

Let $\overrightarrow{\mathcal C_e}$ consist of all finite structures $(A,\leq, E, R)$ such that $(A,\leq)$ is a linear order, $(A,E)$ is a graph, $(A,R)$ is a 3-hypertournament and for every $a,b,c\in A$ with $a<b<c$ we have that $(a,b,c)\in R$ if and only if there are an even number of edges (relation $E$) on $\{a,b,c\}$. Otherwise $(a,c,b)\in R$.

\begin{theorem}\label{thm:pos}
The 4-constrained classes of finite 3-hypertournaments with $S\in\{\{\str C_4\}, \{\str C_4,\str H_4\}, \{\str C_4,\str O_4,\str H_4\}\}$ all have a Ramsey expansion in a finite language. More concretely:
\begin{enumerate}
\item\label{pos:cyc} $\overrightarrow{\mathcal C_{c}}$ is Ramsey.
\item\label{pos:even} $\overrightarrow{\mathcal C_e}$ is Ramsey.
\item\label{pos:all} The class of all finite linearly ordered 3-hypertournaments is Ramsey.
\end{enumerate}
\end{theorem}
We remark that these expansions can be shown to have the so-called \emph{expansion property} with respect to their base classes, which means that they are the optimal Ramsey expansions (see e.g. Definition~3.4 of~\cite{Hubicka2016}).
\begin{proof}
In $\overrightarrow{\mathcal C_{c}}$, $R$ is definable from $\leq$ and we can simply use Ramsey's theorem. Similarly, in $\overrightarrow{\mathcal C_e}$, $R$ is definable from $\leq$ and $E$, hence part~\ref{pos:even} follows from the Ramsey property of the class of all ordered graphs~\cite{Nevsetvril1977b}.

To prove part~\ref{pos:all}, fix a pair of finite ordered 3-hypertournaments $\str A$ and $\str B$ and use the Ne\v set\v ril--R\"odl theorem~\cite{Nevsetvril1977b} to obtain a finite ordered holey 3-hypertournament $\str C'$ such that $\str C' \longrightarrow (\str B)^\str A_2$. The holes in $\str C'$ can then be filled in arbitrarily to obtain a linearly ordered 3-hypertournament $\str C$ such that $\str C \longrightarrow (\str B)^\str A_2$.
\end{proof}

\section{The $\str H_4$-free case}
Let $\str A = (A,R)$ be a holey 3-hypertournament. We say that $\overbar{\str A} = (A, R')$ is a \emph{completion} of $\str A$ if $R\subseteq R'$ and $\overbar{\str A}$ is an $\str H_4$-free 3-hypertournament. Most of the known Ramsey classes can be proved to be Ramsey by a result of Hubi\v cka and Ne\v set\v ril~\cite{Hubicka2016}. In order to apply the result for $\str H_4$-free 3-hypertournaments, one needs a finite bound $c$ such that whenever a holey 3-hypertournament has no completion, then it contains a substructure on at most $c$ vertices with no completion. (Completions defined in~\cite{Hubicka2016} do not directly correspond to completions defined here. However, the definitions are equivalent for structures considered in this paper.) We prove the following.

\begin{theorem}\label{thm:neg}
There are arbitrarily large holey 3-hypertournaments $\str B$ such that $\str B$ has no completion but every proper substructure of $\str B$ has a completion.
\end{theorem}

This theorem implies that one cannot use~\cite{Hubicka2016} directly for $\str H_4$-free hypertournaments. However, a situation like in Theorem~\ref{thm:neg} is not that uncommon. There are two common culprits for this, either the class contains orders (for example, failures of transitivity can be arbitrarily large in a holey version of posets) or it contains equivalences (again, failures of transitivity can be arbitrarily large). In the first case, there is a condition in~\cite{Hubicka2016} which promises the existence of a linear extension, and thus resolves the issue. For equivalences, one has to introduce explicit representatives for equivalence classes (this is called \emph{elimination of imaginaries}) and unbounded obstacles to completion again disappear.

For $\str H_4$-free hypertournaments neither of the two solutions seems to work. This means that something else is happening which needs to be understood in order to obtain a Ramsey expansion of $\str H_4$-free tournaments. Hopefully, this would lead to new, even stronger, general techniques.

In the rest of the abstract we sketch a proof of Theorem~\ref{thm:neg}.

\begin{lemma}\leavevmode
\begin{enumerate}
\item Let $\str G = (G, R)$ be a holey 3-hypertournament with $G=\{1,2,3,4\}$ such that $(1,3,4)\in R$, $(1,4,2)\in R$ and $\{1,2,3\}$ and $\{2,3,4\}$ are holes. Let $(G,R')$ be a completion of $\str G$. If $(1,2,3)\in R'$, then $(2,3,4)\in R'$.
\item Let $\str G^\neg = (G, R)$ be a holey 3-hypertournament with $G=\{1,2,3,4\}$ such that $(2,4,3)\in R$, $(1,4,2)\in R$ and $\{1,2,3\}$ and $\{1,3,4\}$ are holes. Let $(G,R')$ be a completion of $\str G^\neg$. If $(1,2,3)\in R'$, then $(1,3,4)\notin R'$.
\end{enumerate}
\end{lemma}
\begin{proof}
In the first case, suppose that $(1,2,3)\in R'$. If $(2,4,3)\in R'$, then $(G,R')$ is isomorphic to $\str H_4$. Hence $(2,3,4)\in R'$. The second case is proved similarly.
\end{proof}
Suppose that $\str A=(A,R)$ is a holey 3-hypertournament. For $x,y,z,w\in A$, we will write $xyz\Rightarrow yzw$ if the map $(1,2,3,4)\mapsto(x,y,z,w)$ is an embedding $\str G\to \str A$ and we will write $xyz\Rightarrow \neg xzw$ if the map $(1,2,3,4)\mapsto(x,y,z,w)$ is an embedding $\str G^\neg\to\str A$. Using the complement of $\str G$, we can define $\neg xyz\Rightarrow \neg yzw$, and using the complement of $\str G^\neg$ we can define $\neg xyz\Rightarrow xzw$. This notation can be chained as well, e.g. $xyz\Rightarrow yzw \Rightarrow zwu \Rightarrow \neg zuv$ means that all of $xyz\Rightarrow yzw$, $yzw \Rightarrow zwu$, $zwu \Rightarrow \neg zuv$ are satisfied.

Let $n\geq 6$. We denote by $\str O_n = (O_n, R)$ the holey 3-hypertournament with vertex set $O_n=\{1,\ldots,n\}$ such that
$$123\Rightarrow 234\Rightarrow 345  \Rightarrow \cdots \Rightarrow (n-2)(n-1)n \Rightarrow \neg(n-2)n1\Rightarrow \neg n12 \Rightarrow \neg 123.$$
All triples not covered by these conditions are holes.

\begin{lemma}\label{lem:on}\leavevmode
\begin{enumerate}
\item\label{on:1} There is a completion $(O_n, R')$ of $\str O_n$.
\item\label{on:2} If $(O_n,R')$ is a completion of $\str O_n$, then $(1,2,3) \notin R'$.
\item\label{on:3} For every $v\in O_n\setminus\{1,2,3\}$ there is a completion $(O_n\setminus\{v\},R')$ of the structure induced by $\str O_n$ on $O_n\setminus\{v\}$ such that $(1,2,3)\in R'$.
\end{enumerate}
\end{lemma}
\begin{proof}
For part~\ref{on:1}, observe that every set of four vertices of $\str O_n$ with at least two different subsets of three vertices covered by a relation is isomorphic to $\str G$, $\str G^\neg$ or the complement of $\str G$. It follows that whenever $x,y,z\in O_n$ is a hole such that $x<y<z$, we can put $(x,z,y)$ and its cyclic rotations in $R'$ to get a completion. Part~\ref{on:2} follows by induction on the conditions.

For part~\ref{on:3}, we put $(1,2,3), (2,3,4), \ldots (v-3,v-2,v-1)\in R'$, $(v+1,v+3,v+2), \ldots, (n-2,n,n-1)\in R'$ and $(n-2,n,1), (n,1,2) \in R'$. It can be verified that this does not create any copies of $\str H_4$. A completion of $(O_n, R')$ exists as the class of all finite $\str H_4$-free tournaments has strong amalgamation.
\end{proof}

Similarly, for $n\geq 6$, we define $\str O_n^\neg = (O_n^\neg, R)$ the holey 3-hypertournament with vertex set $O_n^\neg=\{1,\ldots,n\}$ such that
$$\neg 123\Rightarrow \neg 234\Rightarrow \neg 345 \Rightarrow \cdots \Rightarrow \neg (n-2)(n-1)n \Rightarrow (n-2)n1\Rightarrow n12 \Rightarrow 123$$
and there are no other relations in $R$. In any completion $(O_n^\neg, R')$ of $\str O_n^\neg$ it holds that $(1,2,3)\in R'$, in fact, an analogue of Lemma~\ref{lem:on} holds for $\str O_n^\neg$.

\medskip

Let $\str B_n$ be the holey 3-hypertournament obtained by gluing a copy of $\str O_n$ with a copy of $\str O_n^\neg$, identifying vertices $1$, $2$ and $3$. (This means that $\str B_n$ has $2n-3$ vertices.) We now use $\{\str B_n : n\geq 6\}$ to prove Theorem~\ref{thm:neg}.

\begin{proof}[of Theorem~\ref{thm:neg}]
Assume that $(B_n, R')$ is a completion of $\str B_n$. So in particular, it is a completion of the copies of $\str O_n$ and $\str O_n^\neg$. By Lemma~\ref{lem:on} and its analogue for $\str O_n^\neg$, we have that $(1,2,3)\notin R'$ and $(1,2,3)\in R'$, a contradiction. 

Pick $v\in B_n$ and consider the structure $\str B_n^v$ induced by $\str B_n$ on $B_n\setminus\{v\}$. We prove that $\str B_n^v$ has a completion. If $v\notin\{1,2,3\}$, one can use part~\ref{on:3} of Lemma~\ref{lem:on} and its analogue for $\str O_n^\neg$ to complete the copy of $\str O_n$ and $\str O_n^\neg$ (one of them missing a vertex) so that they agree on $\{1,2,3\}$. Using strong amalgamation, we get a completion of $\str B_n^v$. If $v\in\{1, 2, 3\}$, we pick an arbitrary completion of $\str O_n$ and $\str O_n^\neg$, remove $v$ from both of them, and let the completion of $\str B_n$ to be the strong amalgamation of the completions over $\{1,2,3\}\setminus v$.
\end{proof}

The following question remains open.
\begin{question}
What is the optimal Ramsey expansion for the class of all finite $\str H_4$-free hypertournaments? Does it have a Ramsey expansion in a finite language?
\end{question}

%
%

\end{document}